\documentclass[11pt]{amsart}
\usepackage{amsmath,amssymb,latexsym,esint,cite,mathrsfs}
\usepackage{verbatim,wasysym}
\usepackage[left=3cm,right=3cm,top=2.9cm,bottom=3cm]{geometry}

\usepackage{microtype}
\usepackage{color,enumitem,graphicx}
\usepackage[colorlinks=true,urlcolor=blue, citecolor=red,linkcolor=blue,
linktocpage,pdfpagelabels, bookmarksnumbered,bookmarksopen]{hyperref}
\usepackage[hyperpageref]{backref}
\usepackage[english]{babel}

\newtheorem{lemma}{Lemma}[section]
 
\newtheorem{proposition}[lemma]{Proposition}
\newtheorem{theorem}[lemma]{Theorem}
\newtheorem{ex}[lemma]{Example}

\theoremstyle{definition}

\newtheorem{remark}[lemma]{Remark}

\newcommand{\R}{\mathbb{R}}
\newcommand{\N}{\mathbb{N}}

\newcommand{\eps}{\varepsilon}
\newcommand{\be}{\begin{equation}}
\newcommand{\ee}{\end{equation}}

\DeclareMathOperator{\Div}{div}

\numberwithin{equation}{section}

\title[Concavity properties for solutions to some $p$-Laplace equations]{Concavity properties for solutions  to $p$-Laplace \\ equations with concave nonlinearities}

\author[W.\ Borrelli]{William Borrelli}
\author[S. \ Mosconi]{Sunra Mosconi}
\author[M.\ Squassina]{Marco Squassina}

\address[W.\ Borrelli]{Dipartimento di Matematica e Fisica
	\newline\indent
	Universit\`a Cattolica del Sacro Cuore
	\newline\indent
	Via della Garzetta 48, I-25133 Brescia, Italy}
\email{william.borrelli@unicatt.it}

\address[S.\ Mosconi]{Department of Mathematics and Computer Science
	\newline\indent
	University of Catania
	\newline\indent
	Viale A. Doria 6, I-95125 Catania, Italy}
\email{sunra.mosconi@unict.it}

\address[M.\ Squassina]{Dipartimento di Matematica e Fisica
	\newline\indent
	Universit\`a Cattolica del Sacro Cuore
	\newline\indent
	Via della Garzetta 48, I-25133 Brescia, Italy}
\email{marco.squassina@unicatt.it}

\thanks{The authors are members of {\em Gruppo Nazionale per l'Analisi Ma\-te\-ma\-ti\-ca, la Probabilit\`a e le loro Applicazioni} (GNAMPA) of the {\em Istituto Nazionale di Alta Matematica} (INdAM)}


\begin{document}

\begin{abstract}
We obtain new concavity results, up to a suitable transformation, for a class of quasi-linear equations
in a convex domain involving the $p$-Laplace operator and a general nonlinearity satisfying  concavity type assumptions.
This provides an extension of results previously known in the literature only for the torsion and
the eigenfunction equations. In the semilinear case $p=2$ the results are already new since they include new admissible
nonlinearities.
\end{abstract}

\maketitle

	\begin{center}
		\begin{minipage}{9cm}
			\small
			\tableofcontents
		\end{minipage}
	\end{center}

\section{Introduction}

\subsection{Overview}
Convexity properties of solutions to elliptic partial differential equations in a convex domain are a fascinating subject. 
One of the first results in this direction can be traced back to the work of Brascamp and Lieb  in 1976  \cite{BL}. They 
proved that the logarithm function applied to the first eigenfunction $\phi_1$ of the Laplace operator $-\Delta$ with zero Dirichlet
boundary conditions in a convex domain is concave. It is readily seen that 
$\phi_1$ itself \emph{is never} 
concave in any convex domain, thus considering a transformation 
of the solution is necessary. Previously, in 1971, Makar-Limanov \cite{Makar} proved 
that if $u>0$ is the  solution to the torsion equation $\Delta u+1=0$ in a convex planar domain, then $\sqrt{u}$ is concave.
Years later, at the beginning of the eighties, Korevaar \cite{korevaar,kore2} and Kennington \cite{kennington} derived these results from some 
general convexity properties (see also \cite{kaw,CS,CaFri} for related seminal works in those years). 

More precisely, given $\Omega\subset \R^N$ convex and a function $u$ on $\bar\Omega$, 
these are maximum principles for 
$$
(y,z,\lambda)\mapsto  
	u(\lambda \,y			+ (1-\lambda)\, z)
	-\lambda \, u(y)-(1-\lambda)\, u(z),
$$
for $y,z\in\bar\Omega$ and $\lambda\in [0,1]$ since
positivity (negativity) in $\bar\Omega\times\bar\Omega\times [0,1]$ is equivalent to concavity (convexity) of $u$.  As a byproduct, some results about
concavity of positive solutions of semilinear problems can be obtained.
For instance, if $N\geq 2$, $q\in (0,1)$, $\Omega$ is  convex and $u$ is a solution to
$$
\begin{cases}
	-\Delta u=u^q, & \text{in $\Omega$,}   \\
	u>0, & \text{in $\Omega$,}   \\
	u=0, & \text{on $\partial\Omega$\,,}   
\end{cases}
$$
then $u^{(1-q)/2}$ is concave in $\bar\Omega$. 
Roughly speaking, some form of concavity on the nonlinear term forces a suitable power of the solution to be concave. In 1987 Sakaguchi treated, via
a suitable approximation argument to handle lack of regularity of the solutions,
 the following problems
involving the $p$-Laplacian operator  
\[
\begin{cases}
	-\Delta_p u=\lambda \, u^{p-1}, & \text{in $\Omega$,}   \\
	u>0, & \text{in $\Omega$,}   \\
	u=0, & \text{on $\partial\Omega$\,,}   
\end{cases}
\qquad \qquad
\begin{cases}
	-\Delta_p u=1, & \text{in $\Omega$,}   \\
	u>0, & \text{in $\Omega$,}   \\
	u=0, & \text{on $\partial\Omega$\,,}   
\end{cases}
\]
where in the former, $\lambda$ is (necessarily) the first eigenvalue. In \cite{Saka}, he indeed proved that $\log u$ and $u^{(p-1)/p}$ are concave, respectively. 

\subsection{Main results}
Given the quasi-linear elliptic problem

\be
\label{eq:p-laplace}
\begin{cases}
-\Delta_p u=f(u), & \text{in $\Omega$,}   \\
u>0, & \text{in $\Omega$,}   \\
u=0, & \text{on $\partial\Omega$,}   
\end{cases}
\ee
for $p>1$ and $\Omega$ convex and bounded, our focus is to find the most general reaction term $f$ ensuring that $u$ is {\em quasi-concave}, i.\,e.\,, its super-level sets $\{u>k\}$ are convex for any $k$ and to highlight the interplay between the reaction $f$ and suitable concavity properties of $u$.  

A natural method to obtain quasi-concavity is to find an increasing function $\varphi:\R_+\to \R$ such that $\varphi(u)$ is concave.
In this respect, by slightly modifying the proof of Sakaguchi \cite{Saka}, we find the following

\begin{theorem}\label{Sth}
Let $\Omega\subset\R^N$ be a bounded convex domain with $C^2$ boundary, $f:\R_+ \to  [0, +\infty)$ be H\"older continuous with 
\[
M=\inf\big\{t>0: f(t)=0\big\}.
\]
and $u\in W^{1,p}_0(\Omega)$ solve  \eqref{eq:p-laplace}. If
\begin{enumerate}
\item
 $t\mapsto f(t)/t^{p-1}$ is non-increasing,
\item
$t\mapsto e^{(p-1) t}/f(e^t)$ is convex on $(-\infty, \log M)$,
\end{enumerate}
then $\log u$ is concave.
\end{theorem}

 Notice that, given any concave increasing $h:\R\to \R$, the concavity of $\varphi(u)$ (called $\varphi$-concavity of $u$) implies that $h(\varphi(u))$ is concave as well. Supposing both $h$ and $\varphi$ are smooth, by computing the second derivative of $h\circ \varphi$ one sees that $h\circ \varphi$ is "more concave" than $\varphi$, so that our interest  is to determine the "less concave" $\varphi$ such that $\varphi(u)$ is concave, as reasonable measure  of the optimal quasi-concavity of $u$. For example,  the less concave increasing functions $\varphi$ are clearly the affine ones, and in this case $\varphi$-concavity reduces concavity.
 
To this end, suppose $f\in C(\R_+, [0, +\infty))$ and let
\[
F(t)=\int_0^t f(u)\, du.
\]
We will  consider the function $\varphi:(0, +\infty)\to \R$ defined by
\be
\label{defvarphi}
\varphi(t):=\int_1^t\frac{1}{F^{1/p}(\tau)}\, d\tau\,
\ee
as detailed in the following theorem, which is the main result of the paper.

\begin{theorem}\label{Mth}
Let $\Omega\subset\R^N$ be a bounded convex domain with $C^2$ boundary, $f:\R_+ \to  [0, +\infty)$ be H\"older continuous with 
\[
M=\inf\big\{t>0: f(t)=0\big\}.
\]
and $u\in W^{1,p}_0(\Omega)$ solve  \eqref{eq:p-laplace}. If
\begin{enumerate}
\item
If $F^{1/p}$ is concave,
\item
$F/f$ is   convex on $(0, M)$,
\end{enumerate}
then $\varphi(u)$ is concave, where $\varphi$ is defined in \eqref{defvarphi}.
 \end{theorem}

Let us make some remarks on these two statements.

\begin{remark}(Motivations) 

A conjecture of Lions \cite{lions} going back to the eighties states that any solution of \eqref{eq:p-laplace} for $p=2$ is quasi-concave, i.\,e.\, has convex super-level sets, as long as $\Omega$ is convex and $f\ge 0$. The latter statement (while true in the ball thanks to the celebrated Gidas-Ni-Nirenberg symmetry result) has been recently disproved in \cite{hns}, thus revamping the question wether it is possible to  select a large class of reactions $f$ for which Lions' statement turns out to be true.   

The subject of the optimal concavity properties of $u$ has been the object of   recent research in the restricted framework of optimal {\em power concavity} (i.\,e.\, concavity properties of powers of $u$, for the highest possible power). It is known \cite[item (iii) in Thm. 6.2 ]{kennington} that if $f(t)\equiv 1$ and $p=2$, the solution of \eqref{eq:p-laplace} is such that $u^{q}$ is concave for $q\le 1/2$, while for $q>1/2$ there are convex domains for which $u^q$ fails to be concave. However this phenomenon is known to happen only for convex domains having corners, while in the ball the torsion function is actually concave (thus, the optimal concavity exponent of the ball is $1$).
See \cite{hnst} and the literature therein for further references on the general problem of linking the optimal concavity exponent with the smoothness of the domain.
\end{remark}

\begin{remark}(Relations between theorem \ref{Sth} and \ref{Mth})

Condition  {\em (1)} of Theorem \ref{Mth} implies {\em (1)} of Theorem \ref{Sth}, but not the opposite,  see Lemma \ref{lemma21} and Remark \ref{remconv}.
For $p=2$ the reaction 
\[
f(u)=1+\sqrt{u}
\]
 satisfies the assumptions of Theorem \ref{Sth} but not {\em (2)} of Theorem \ref{Mth}. However, the  concavity property asserted in Theorem \ref{Mth} is in general stronger than the $\log$-concavity provided by Theorem \ref{Sth}. Indeed, if $\psi$ denotes the inverse of the function $\varphi$ defined in \eqref{defvarphi}, from 
\[
\log u=\log \psi(\varphi(u)),
\]
we see that, as long as $t\mapsto \log\psi(t)$ is concave, $\varphi$-concavity of $u$ implies its $\log$ concavity. Lemma \ref{lemmapsi} below shows, under assumption {\em (1)} of Theorem \ref{Mth}, $\log\psi$ is always concave. That $\varphi$-concavity is strictly stronger than $\log$ concavity is readily seen considering power-concavity, see below.
\end{remark}
\begin{remark}(Applications)

 If $f\in C^1(\R_+)$, condition {\em (1)} and {\em (2)} of Theorem \ref{Mth} read
 \[
 \frac{F\, f'}{f^2}\le 1-\frac{1}{p}, \qquad \left( \frac{F\, f'}{f^2}\right)'\le 0
 \]
 on $(0, M)$, so that actually the assumptions are equivalent to 
\[
 \frac{F\, f'}{f^2}\quad \text{non-increasing on $(0, M)$ \qquad and }\qquad \lim_{t\to 0^+} \frac{F\, f'}{f^2}\le 1-\frac{1}{p}
\]
The assumptions in Theorem \ref{Mth} are easily verified for $f(t)=t^q$ with $q\in [0,p-1)$, giving, for its first instance, the concavity of $u^{(p-q-1)/p}$ (briefly called $\alpha$-concavity of $u$ for $\alpha=(p-q-1)/p$)
of the solution $u$ of
\[
\begin{cases}
-\Delta_p u=u^q, & \text{in $\Omega$,}   \\
u>0, & \text{in $\Omega$,}   \\
u=0, & \text{on $\partial\Omega$.}   
\end{cases}
\]
However, already in the semilinear case $p=2$, several meaningful reactions related to natural entropy models can  be treated. We refer to section \ref{sec2} for such examples of non-power like reactions.
\end{remark}

\begin{remark}(Assumptions on $f$)

The first assumptions of both Theorem \ref{Sth} and Theorem \ref{Mth} strengthen as $p$ decreases, meaning that if the  they hold for some $p$, then they do for any $q\ge p$ as well.

The non-negativity hypothesis on the reaction $f$ is not essential and we made it to have a cleaner statement. It is possible to deal with functions obeying 
\[
f(t)\le 0\quad \text{ for all }\quad  t\ge M:=\inf\{t:f(t)\le 0\},
\]
 since in this case an application of the weak maximum principle shows that $u\le M$ (more details on this a-priori estimate can be found in Example \ref{exlin} or at the beginning of Section \ref{sec4}). This, in turn, allows to require the whole assumption in  {\em (1)} of Theorem \ref{Mth} to be fulfilled on the interval $(0, M)$ only. This restriction can sometimes be useful, as it allows to prove quasi-concavity of solutions to simple problems such as
\[
\begin{cases}
-\Delta u+u=1, & \text{in $\Omega$,}   \\
u>0, & \text{in $\Omega$,}   \\
u=0, & \text{on $\partial\Omega$,}   
\end{cases}
\]
in Example \ref{exlin}. 

Let us also mention that the H\"older continuity of $f$ can be weakened to Dini-continuity, since this suffice to obtain $C^2( \Omega)$ of solution of regular elliptic quasilinear problems with reaction $f(u)$.
\end{remark}

\begin{remark}(Comparison with previous results)

Theorem \ref{Sth} is folklore and we stated it only for completeness, since its proof follows the same approximation procedure of \cite{Saka}.
Regarding Theorem \ref{Mth}, the cases $f(t)\equiv 1$ and $f(t)=t^{p-1}$ have been treated in \cite{Saka}.  The power case $f(t)=t^q$ with $0<q<p-1$ is essentially contained in \cite{huaxshu}.  

General reaction terms have been previously considered in \cite{lions, korevaar} in the semilinear case $p=2$, obtaining $\log$-concavity of the positive solution of
\[
-\Delta u=f(u), \qquad u=0 \quad \text{on $\partial \Omega$}
\]
for  $f\in C^1\big([0, +\infty), [0, +\infty)\big)$ fulfilling
\be
\label{lions}
t\mapsto f(t)/t \quad \text{non-increasing\qquad and } \qquad f''(t)\, t^2-f'(t)\, t+f(t)\le 0.
\ee
The second condition above fails for powers $f(t)=t^q$ and is equivalent to the {\em concavity} of $g(t)= f(e^t)/e^t$. Hypothesis {\em (2)} in Theorem \ref{Sth} (in the case $p=2$) is instead equivalent to the so-called {\em harmonic concavity} of $g$. It is well known that any concave function is harmonic concave, but the opposite is not true. For example, both Theorems \ref{Sth} and \ref{Mth} apply in the semilinear case $p=2$ with $f(t)=t^q$ with $q<1$ (the latter providing, as already remarked, a stronger conclusion than the former), but \eqref{lions} fails.

The monograph \cite{kaw} reviews  most of the techniques available to prove quasi-concavity of solutions to elliptic differential equation up to the mid-eighties. Apart from Korevaar-Kennington approach, another fruitful method is to consider the concave envelope, as done in the seminal work \cite{ALL} in the framework of viscosity solutions of fully nonlinear equations. This technique has been applied in \cite{CF,  CF2, juutinen, juutinen2, hoeg, kuhn}. Other close approaches are the quasi-concave envelope method and its modifications \cite{BS,CS1,CS2} and the microscopic concavity principle initiated in \cite{CaFri} and developed in \cite{korlew, bianguan}. The latter has been employed in \cite{lin} to prove that any solution in a convex domain of the plane of \eqref{eq:p-laplace} for $p=2$ and $f(t)=t^q$ with $q>1$ is $(q-1)/2$-concave. Unfortunately, none of these methods easily apply to \eqref{eq:p-laplace} unless additional a-priori (and unexpected) regularity is assumed on the solution $u$.

\end{remark}
 
 \subsection{Outline of the proof}
 The choice of the transformation \eqref{defvarphi} can be motivated, at least heuristically, by the following argument.

 Let $u$ be a solution to \eqref{eq:p-laplace}. Consider an increasing invertible function $\varphi:\R_+\to \R$ with inverse $\psi$. We set
  \[
  v:=\varphi(u)
  \]
 and compute the equation solved by $v$ from the one for $u$, to obtain
  \be\label{eq:generalveq}
  -\left(\Delta v+\frac{p-2}{\vert\nabla v\vert^2}\sum^N_{i,j=1}\partial_i v\partial_j v \partial^2_{ij}v\right)=\frac{f(\psi(v))}{\psi'(v)^{p-1}}\, \vert\nabla v\vert^{2-p}+(p-1)\, \frac{\psi''(v)}{\psi'(v)}\, \vert \nabla v\vert^2=0
  \ee
  In order to simplify \eqref{eq:generalveq} we force a factorisation of the dependencies on $v$ and on $\nabla v$ by requiring
  \be\label{eq:condition}
 \frac{f(\psi)}{(\psi')^{p-1}}=p\, \frac{\psi''}{\psi'}
  \ee
(the constant $p$ is irrelevant of course, we chose that for convenience). This amounts, by integration, to
\[
  \psi'(s)=F(\psi(s))^{1/p}\qquad s>0
  \]
  where $F$ is the primitive of $f$ that vanishes at $\psi(0)$. This condition indeed gives $\varphi$ given in \eqref{defvarphi}. 
  
  By \eqref{eq:condition} we can rewrite \eqref{eq:generalveq} in the more tractable form:
\be\label{eq:newveq}
 \Delta v+\frac{p-2}{\vert\nabla v\vert^2}\sum^n_{i,j=1}\partial_i v\partial_j v \partial^2_{ij}v+b(v,\nabla v)=0
\ee
where 
\[
b(v,\nabla v):= \frac{\psi''(v)}{\psi'(v)}\, \left( p\, \vert\nabla v\vert^{2-p} +(p-1)\, \vert\nabla v\vert^2\right)\, ,
\]
and we can require the assumptions  on \cite[Theorem 3.1]{kennington} to apply Kennington theorem. The latters involve only the ratio $\psi''/\psi'$ and, as it turns out, are equivalent to {\em (1)} and {\em (2)} of Theorem \ref{Mth}.

Of course the problem is that equation \eqref{eq:newveq} lacks sufficient smoothness of both the right and the left hand side to employ directly Kennington's theorem, and the derivation depicted above seems quite rigid. The issue is then to find suitable regularisations of the original equation so that the previous computations makes sense and the requirements on the term $\psi''/\psi'$ are stable in the limit.

Luckily, the expected  condition {\em (1)} in Theorem \ref{Mth} also implies  uniqueness of the solution of the original equation \eqref{eq:p-laplace}. This is proved in Theorem \ref{BO} through a refined version of the celebrated Brezis-Oswald theorem \cite{brezis-oswald}. This preliminary  result is pivotal in ensuring the effectiveness of any approximation.

 It turns out that a good approximation is given by the minimisers $u_\eps$ of the  functional
 \[ 
 I_\eps(u):=\frac{1}{p}\int_\Omega\big (\eps\, F(u)^{2/p}+\vert \nabla u\vert^2\big)^{p/2}\,dx-\int_\Omega F(u)\,dx\,,
 \]
as $\eps \to 0$. Indeed, the aforementioned uniqueness ensures that $u_\eps\to u$ in all reasonable senses and, for sufficiently small $\eps$, we are able to prove that $\varphi(u_\eps)$ is actually concave. The claimed concavity of $\varphi(u)$ then follows.
  
  \subsection{Structure of the paper and notations}

In Section \ref{sec2} we provide examples and application of Theorem \ref{Mth}. Section \ref{sec3} is devoted to some preliminary material: we first recall the main tools we will use, essentially contained in \cite{korevaar, kennington}, then prove some properties of the functions $f$ and $\varphi$, and conclude the section with the aforementioned Brezis-Oswald-type result. In Section \ref{sec4} we focus on the proof of Theorems \ref{Sth} and \ref{Mth}, dividing it in several steps. 
\medskip

{\em Notations:}
In the following $\Omega$ will always denote a bounded open convex domain, with {\em interior} normal usually denoted by $n$. The norm in $L^p(\Omega)$ of a function $u:\Omega\to \R$ will be denoted by $\|u\|_p$ for any $p\in [1, +\infty]$. We will use the same notation also for vector valued functions. Constants may change in value from line to line without changing symbol, as long as their dependance is clear from the context.

\begin{center}
{\bf Acknowledgements}
\end{center}
S. Mosconi is partially supported by project PIACERI - Linea 2 and 3 of the University of Catania.
  Part  of  the  paper  was  developed  during  a  visit  of the  second  author  at  the  Department  of  Mathematics  and  Physics of the Catholic University  of Sacred Hearth, Brescia, Italy.  
  The  hosting  institution  is  gratefully  acknowledged.

 \section{Examples and applications}\label{sec2}
 In this section we provide some examples of application of Theorem \ref{Mth} which were not available via known results.
 \begin{ex}\label{exentr}
The following examples provide non power-like functions $F$ fulfilling the assumptions of Theorem \ref{Mth} for $p=2$ (and thus for any $p\ge 2$). 
\begin{itemize}
\item
Consider the function $F:\R_+\to \R$ given by
\[
F(t)=t\, \log(1+t).
\]
 Then 
\[
\big( F^{1/2}\big)''(t)=- \frac{F(t)^{-\frac{3}{2}}}{4 (t+1)^2} \left[(t-\log(t+1))^2 +  \big((t + 1)^2-1\big) \log^2(t + 1) \right]
\]
which is non-positive for $t\ge 0$. Moreover
\[
\big(F/f\big)''(t)=\frac{(t+1)\big(t\, (t+4)-2\, \log(t+1)\big)\log(t+1)-t^2(t+2)}{(t+1)\big(t+(t+1)\, \log(t+1)\big)^3},
\]
so that, in order to prove that $F/f$ is convex, it suffices to consider the numerator $N(t)$. Using $\log (t+1)\le t$ we get
\[
N(t)\ge g(t):=(t+1)\big(t\, (t+4)-2\, t)\big)\log(t+1)-t^2(t+2), 
\]
and the auxiliary function $g$ obeys 
\[
g^{(k)}(0)=0 \quad \text{for $k=0,1, 2$}, \qquad g^{(3)}(0)=6, \qquad g^{(4)}(t)=\frac{2\, (3\, t^2+6\, t+2)}{(t+1)^3}.
\]
Applying Taylor's theorem with integral remainder, we thus see that $g\ge 0$ on $[0, +\infty)$, proving the claim.
\item
The logarithmic entropy function 
\[
F(t)=(t+1)\, \log(t+1)-t
\]
for $t\in \R_+$  verifies
\[
\big( F^{1/2}\big)''(t)=- \frac{F(t)^{-\frac{3}{2}}}{4\, (t+1)}\left[2\, t+(t+1)\big(\log(t+1)-2\big)\log(t+1)\right]=:- \frac{F(t)^{-\frac{3}{2}}}{4\, (t+1)}g(t)
\]
and the function $g(t)$ is nonnegative, since
\[
g'(0)=0, \qquad g'(t)=\log^2(t+1),
\]
implying that $F^{1/2}$ is concave. Finally,
\[
\big(F/f\big)''(t)=\frac{(t+2)\log(t+1)-2\, t}{(t+1)^2\log^3(t+1)}
\]
whose numerator $N(t)$ fulfills
\[
N(0)=N'(0)=0, \qquad N''(t)=\frac{t}{(t+1)^2}
\]
and is therefore nonnegative.
\item
In general, observe that if a function $F$ satisfies the assumptions of Theorem \ref{Mth} for $p=2$ then, given $q>1$, $F^{q/2}$ satisfies the assumptions for $p=q$ as
\[
(F^{q/2})^{1/p}= F^{1/2}\,,\qquad \frac{F^{q/2}}{(F^{q/2})'}=\frac{2}{q}\frac{F}{f}\,.
\]
\end{itemize} 
\end{ex}

We conclude with an example showing what is the r\^ole of the restriction to the image of the conditions in Theorem \ref{Mth}

\begin{ex}\label{exlin}
Consider the problem
\be
\label{ex1}
\begin{cases}
-\Delta u+u=1 &\text{in $\Omega$}\\
u=0&\text{on $\partial\Omega$}\\
u>0&\text{in $\Omega$}
\end{cases}
\ee
for a convex $\Omega$. A unique, non-negative solution  is immediately obtained via minimisation of the strictly convex functional
\[
J(v)=\int_\Omega |\nabla v|^2+v^2-2\, v\, dx,
\]
but Theorem \ref{Mth} does not naively apply, since the reaction term fails to be non-negative on $\R_+$ and the concavity of $\sqrt{t-t^2/2}$ can possibly hold at most for $t\in [0, 2]$. A way to apply Theorem \ref{Mth} in this setting is to consider the truncation
\[
f(t)=(1-t)_+
\]
and the associated problem 
\be
\label{ex2}
\begin{cases}
-\Delta u=f(u)&\text{in $\Omega$}\\
u= 0&\text{on $\partial\Omega$}\\
u>0&\text{in $\Omega$}
\end{cases}
\ee
The corresponding solution exists, is $C^2(\overline{\Omega})$, and it obeys
\[
\|u\|_\infty\le 1.
\]
Indeed, if the maximum of $u$ is greater than $1$, on the open set $A=\{u>1\}$ the function $u-1$ is harmonic in $A$, non-negative and attains a positive maximum, contradicting the weak maximum principle. Therefore, the solution of \eqref{ex2} coincides with the solution of \eqref{ex1}, ensuring $u\le 1$ for the latter. 
It is readily checked that, denoting with $F$ the primitive of $(1-t)_+$, for $t\in (0, 1)$ it holds
\[
\big(F^{1/2}\big)''(t)=-\frac{F^{-3/2}(t)}{4}<0\qquad \big(F/f\big)''(t)=\frac{1}{(1-t)^3}>0
\]
while clearly $F^{1/2}$ is concave on $\R_+$, so that the solution of \eqref{ex1} has convex level sets. Notice that, looking at the original problem \eqref{ex1}, with
\[
 \tilde{f}=1-u, \qquad \tilde{F}=u-\frac{u^2}{2},
 \]
 the function $\tilde{F}/\tilde{f}$ turns out to be concave for $t>1$.
\end{ex}
\section{Preliminaries}\label{sec3}

In this section we collect various results that will be useful in the proof of the Theorem \ref{Mth}.

\smallskip

\subsection{The convexity function}
We recall here the main tools developed by Korevaar and Kennigton to deal with concavity properties of solution to \eqref{eq:p-laplace}.
Given a continuous function $v:\Omega\to \R$ with $\Omega$ convex, its {\em convexity function} $c:\Omega\times\Omega \to \R$ is defined as 
\[
c(x, y)= \frac{v(x)+ v(y)}{2}-v\Big(\frac{x+ y}{2}\Big).
\]
Clearly, $v$ is concave in $\Omega$ if and only if $c\le 0$ in its domain. The main result of Kennington, generalised in \cite{grecoporru}, is the following; that the assumptions {\em (1)} and {\em (2)} below can be checked only on the image of $v$, while not explicitly stated, follows from inspecting the proof of \cite{grecoporru}.

\begin{proposition}[\cite{grecoporru}]\label{propKennington}
Let $\Omega$ be bounded and convex in $\R^N$, $N\ge 2$ and $v\in C^2(\Omega)\cap C^0(\overline{\Omega})$ solve
\[
\sum_{i, j=1}^Na_{i, j}(Dv)\, \partial_{ij} v+b(v, Dv)=0
\]
for $a_{ij}\in C^0(\R^N)$ such that $(a_{ij})$ is uniformly elliptic and $b\in C^0(\R\times\R^N; \R_+)$. If
\begin{enumerate}
\item
$t\mapsto b(t, z)$ is non-increasing on $v(\Omega)$ for all $z\in \R^N$
\item
$t\mapsto b(t, z)$ harmonic concave on $v(\Omega)$ for all $z\in \R^N$.
\end{enumerate}
Then the convexity function of $v$ cannot  attain a positive maximum in $\Omega\times \Omega$.
\end{proposition}
We refer to \cite{grecoporru} for a general definition of harmonic concavity. For our purposes, the function $b$ will be positive on $v(\Omega)$ and in this case harmonic concavity coincides with the convexity of $t\mapsto 1/b(t, z)$.

Korevaar singled out a class of transformations allowing to exclude that the maximum of the convexity function is attained at the boundary $\partial(\Omega\times\Omega)$.  The following proposition has been obtained  in \cite[Lemma 2.1 and 2.4]{korevaar} for $u\in C^2(\overline{\Omega})$.  Recall that a strongly convex set is a smooth convex set such that the principal curvatures of $\partial\Omega$ are positive. Given such an $\Omega$ we set
\[
S_\delta=\{x\in \Omega: \delta/2\le {\rm dist}(x, \partial\Omega)\le 2\,\delta\}, \qquad \Omega_\delta=\{x\in \Omega: \delta< {\rm dist}(x, \partial\Omega) \}.
\]

\begin{proposition}[\cite{korevaar}, Lemma 2.1 and 2.4 ]\label{propKorevaar}
Suppose that $\Omega$ is smooth, bounded and strongly convex, $N$ a neighbourhood of $\partial\Omega$ and $u\in C^1(\overline{\Omega})\cap C^2(N)$ be  such that 
\be
\label{assu}
  u>0\quad \text{in $\Omega$}, \qquad u=0\quad \text{on $\partial\Omega$}, \qquad \frac{\partial u}{\partial n}>0 \quad \text{on $\partial\Omega$}.
  \ee
  If $\varphi\in C^2(\R_+; \R) $ fulfils
 \be
 \label{assKor}
\lim_{t\to 0^+}\varphi'(t)=+\infty  , \qquad \varphi''<0<\varphi' \ \text{near $0$}\qquad  \lim_{t\to 0^+}\frac{\varphi(t)}{\varphi'(t)}= \lim_{t\to 0^+}\frac{\varphi'(t)}{\varphi''(t)}=0,
 \ee
 and $v:=\varphi(u)$ then, for any sufficiently small $\delta$, it holds
\be
\label{k1}
 D^2 v(x)<0\quad \text{ in }\quad S_\delta,
 \ee
 \be
 \label{k2}
 \forall \bar x\in \partial\Omega_\delta\qquad v(x)<L_{v, \bar x}(x)\quad  \text{in }\quad \Omega_\delta\setminus\{\bar x\},
 \ee
  where  $L_{v, \bar x}(x)=v(\bar x)+\nabla v(\bar x)(x-\bar x)$ is the tangent plane to the graph of $v$ at $\bar x$.
  
 Moreover, if \eqref{k1}, \eqref{k2} hold for some given $v\in C^2(\Omega_{\delta/2})$, its convexity function   cannot attain a positive maximum on $\partial(\Omega_\delta\times\Omega_\delta)$.
\end{proposition}

\subsection{The transformation $\varphi$}
Since we are interested in positive solutions of \eqref{eq:p-laplace}, we will assume henceforth that $f$ is extended to $\R$ as an even function, so that $F$ is odd.
First observe that being $F^{1/p}$ concave on $[0, +\infty)$, it has sublinear growth, implying the estimate
\be
\label{growthF}
F(t)\le C\, (1+t^p), \qquad t\ge 0.
\ee
Applying again the concavity assumption we also infer
\begin{equation}
\label{growthf}
f(t)\le C\,(1+t^{p-1}),\qquad t\ge 0
\end{equation}
therefore solutions of \eqref{eq:p-laplace} are critical points of the $C^1$ functional on $W^{1,p}_0(\Omega)$ defined by 
\begin{equation}
\label{defJ}
J(u)=\int_{\Omega} \frac{1}{p}\, |\nabla u|^p-F(u)\, dx.
\end{equation}
A more refined and useful estimate than \eqref{growthf} is contained in the following lemma.

\begin{lemma}\label{lemma21}
Let $F:[0, +\infty)\to [0, +\infty)$ be differentiable and such that $F(0)=0$. If $F^{1/p}$ is concave for some $p\ge 1$, then $t\mapsto F'(t)/t^{p-1}$ is non-increasing.
\end{lemma}

\begin{proof}
Let $G(t)=F^{1/p}(t)$, and observe that $G'\ge 0$ on $[0, +\infty)$, since $G'(t_0)<0$ implies by concavity 
\[
G(t)\le G(t_0)+G'(t_0)\, (t-t_0)\quad \to\quad  -\infty,
\]
contradicting $G\ge 0$. If
\[
H(s)=G^p(s^{1/p}),
\]
we claim that  $H$ is concave. Indeed,
\[
H'(s)=G^{p-1}(s^{1/p})\, G'(s^{1/p})\, s^{(1-p)/p}=\Big( \frac{G(s^{1/p})}{s^{1/p}}\Big)^{p-1} G'(s^{1/p})
\]
and $s\mapsto H'(s)$ is non-increasing on $\R_+$ if and only if so is $t\mapsto H'(t^p)$, i.e., if and only if 
\begin{equation}
\label{temp}
t\mapsto \Big( \frac{G(t)}{t}\Big)^{p-1} G'(t)
\end{equation}
is non-increasing. Since both 
\[
t\mapsto G'(t)\qquad \text{and}\qquad 
t\mapsto \frac{G(t)}{t}=\frac{G(t)-G(0)}{t-0}
\]
are non-negative, and non-increasing by the concavity of $G$, so is \eqref{temp}, being the product of non-negative, non-increasing functions.
It follows that $H(s)=F(s^{1/p})$ is concave, so that its derivative is non-increasing. Since
\[
H'(s)=\frac{1}{p}\, F'(s^{1/p})\, s^{(1-p)/p},
\]
gives the claim by the monotone increasing change of variable $t=s^{1/p}$.
\end{proof}
\begin{remark}\label{remconv}
The opposite implication in the previous assertion fails to be true.  When $p=2$, the function
\[
F(t)=\sqrt{1+t}+t^2, \qquad t\ge 0
\]
is such that $F'(t)/t$ is non-increasing on $\R_+$, but $F^{1/2}$ is convex for $t\ge 0$.
\end{remark}

We next provide the elementary proof of a property, mentioned in the introduction, of the inverse $\psi=\varphi^{-1}$  of the function $\varphi$ given \eqref{defvarphi}. In that framework, we  apply the lemma with $G=F^{1/p}$.

\begin{lemma}\label{lemmapsi}
Suppose $G\in C^0([0, +\infty); \R_+)$ is concave with $G(0)=0$ and let
\[
\varphi(t)=\int_1^t\frac{1}{G(\tau)}\, d\tau, \qquad \psi(s)=\varphi^{-1}(s).
\]
Then $s\mapsto \log\psi(s)$ is concave.
\end{lemma}

\begin{proof}
We compute
\[
\left(\log\psi\right)'=\frac{\psi'}{\psi}=\frac{G(\psi)}{\psi}
\]
and, since  $\psi$ is increasing, the claim is equivalent to the fact that $t\mapsto G(t)/t$ is non-increasing. But this follows from the assumed concavity of $G$ together with $G(0)=0$.
\end{proof}

In order to apply Proposition \ref{propKorevaar}, we point out the following.

\begin{lemma}\label{lemmavarphi}
Let $f\in C^0(\R_+, [0, +\infty))$ with
\[
M=\inf\big\{t>0:f(t)=0\big\}>0.
\]
If the corresponding $F$ is such that $F^{1/p}$ is concave and $F/f$ is convex in $(0, M)$, the function 
\[
\varphi(t)=\int_1^tF(s)^{-1/p}\, ds, \qquad \mbox{where $F(s)=\int_0^s f(\tau)\, d\tau, \qquad t, s>0$}
\]
fulfils \eqref{assKor} on $(0, M)$. 
\end{lemma}

\begin{proof} 
By the definition of $M$, for any $t\in (0, M)$ it holds
\[
\varphi''(t)=-\frac{f(t)}{p\, F(t)^{1+1/p}}<0, 
\]
and from $\varphi'(t)=F(t)^{-1/p}$ and $F\ge 0$ we readily have $\varphi'(t)\to +\infty$ for $t\to 0^+$.
  By construction $F(t)\le C\, t$ in $(0, M)$ and from the concavity of $F^{1/p}$ we infer $F^{1/p}(t)\ge c\, t$ for $t\in (0, M)$, $c>0$. Therefore, for  $t\in (0, M)$,
\[
0\le \frac{\varphi(t)}{\varphi'(t)}=F^{1/p}(t)\int_t^1\frac{1}{F^{1/p}(t)}\, dt \le \frac{C^{1/p}}{c}\, t^{1/p}\, \log t\to 0.
\]
It remains to prove that
\[
\lim_{t\to 0^+} \frac{\varphi'(t)}{\varphi''(t)}=\lim_{t\to 0^+} p\, \frac{F(t)}{f(t)}=0.
\]
By the convexity of $F/f$ the  limit exists, as does, by monotonicity, the limit
\[
l=\lim_{t\to 0^+}\varphi(t)\le 0.
\]
Hence the claim follows from the previous limit by de l'H\^opital rule, either applied to $\varphi/\varphi'$ if $l=-\infty$   or to $(\varphi-l)/\varphi'$  if $l$ is finite.
\end{proof}
\subsection{A Brezis-Oswald type result}
It is known that the functional $J$ \eqref{defJ} is convex in the variable $w=u^p$ and this implies the monotonicity of the operator $w\mapsto(-\Delta_p w^{1/p})/w^{(p-1)/p}$, as first remarked for $p=2$ by Benguria-Brezis-Lieb \cite{bbl}, see also \cite{brasco-franzina}.

An useful consequence (see e.g. \cite{diaz-saa}) of this convexity property is the following
\begin{lemma}\label{lem:monotonicity}
For $i=1,2$ let $w_i\in L^\infty(\Omega)\cap W^{1,p}(\Omega)$ be such that 
\[
w_i>0, \qquad \Delta_p w_i \in L^\infty(\Omega), \qquad w_1-w_2\in W^{1,p}_0(\Omega), \qquad w_2/w_1, \ w_1/w_2\in L^\infty(\Omega) .
\]
Then there holds
\be\label{monotonicity}
\int_\Omega \left(\frac{-\Delta_p w_1}{w^{p-1}_1}-\frac{-\Delta_p w_2}{w^{p-1}_2}  \right)(w_1^p-w_2^p)\,dx\ge 0
\ee
\end{lemma}
 
We also recall the definition of the first eigenvalue of the $p$-laplacian
\[
\lambda_{1,p}:=\inf \left\{\int_\Omega \vert\nabla v\vert^p\,dx \,:\, v\in W^{1,p}_0(\Omega)\,,\,\int_\Omega\vert v\vert^p\,dx=1 \right\}.
\]

The next Proposition is a Brezis-Oswald \cite{brezis-oswald} type of result. The proof follows D\'iaz and Sa\'a \cite{diaz-saa} but here we deal with the case when $f(t)/t^{p-1}$ is only monotone and not (as assumed in \cite{diaz-saa}) \emph{strictly monotone}. Notice that it holds, more generally, when $\Omega$ is connected but not necessarily convex.

\begin{proposition}
\label{BO}
Let $\Omega\subseteq \R^N$ be connected with $C^2$ boundary,  and let $f\in C^0(\R_+, \R)$ be such that $t\mapsto f(t)/t^{p-1}$ is non-increasing on $(0, +\infty)$. If $u\in W^{1,p}_0(\Omega)\cap C^{1, \alpha}(\overline{\Omega})$ solves \eqref{eq:p-laplace}, then, either  $u$ is a $\lambda_{1,p}$-eigenfunction, or 
\begin{equation}
\label{basicf}
\lim_{t\to +\infty} \frac{f(t)}{t^{p-1}}<\lambda_{1,p}<\lim_{t\to 0^+} \frac{f(t)}{t^{p-1}}
\end{equation}
and
$u$ is the unique solution of \eqref{eq:p-laplace}, which also minimizes $J$ on $W^{1,p}_0(\Omega)$ where, in the definition of $J$, $f$ is evenly extended on $\R$. 
\end{proposition}

\begin{proof}

Suppose that $u$ is not a $\lambda_{1,p}$-eigenfunction. Then, being positive in $\Omega$, it is not an eigenfunction at all, so that we may assume that
\begin{equation}
\label{pa}
\text{$f$ is not of the form $k\, t^{p-1}$ on $[0, \|u\|_\infty]$, }
\end{equation}
By \cite[Theorem 2]{diaz-saa} the solvability of \eqref{eq:p-laplace} implies that
\begin{equation}
\label{p1}
\lim_{t\to +\infty} \frac{f(t)}{t^{p-1}}=:\mu_\infty\le\lambda_{1, p}.
\end{equation}
Indeed, observe that by the monotonicity assumption on $t\mapsto f(t)/t^{p-1}$, it holds
\begin{equation}
\label{pl}
\frac{f(u)}{u^{p-1}}\ge \frac{f(\|u\|_\infty)}{\|u\|_\infty^{p-1}}\ge \mu_\infty.
\end{equation}
Let   $\varphi\in C^{1, \alpha}(\overline{\Omega)}$ be a positive $\lambda_{1,p}$-eigenfunction. Any positive multiple of $\varphi$ is still an eigenfunction so that, using  the Hopf Lemma and $C^{1, \alpha}(\overline{\Omega})$ of both $u$ and $\varphi$,   we can choose $k>0$ such that  $ k\, \varphi>u$ on $\Omega$. Notice that all the assumption of Lemma \ref{lem:monotonicity} hold for $w_1= k\, \varphi$ and $w_2=u$ on $\Omega$, so that by \eqref{monotonicity} and  \eqref{pl}, we obtain 
\[
\begin{split}
0&\leq\int_{\Omega}\left(\frac{-\Delta_p (k\, \varphi)}{(k\, \varphi)^{p-1}}-\frac{-\Delta_p u}{u^{p-1}} \right)\big((k\, \varphi)^p-u^p\big)\, dx=\int_{\Omega}\left(\lambda_{1, p}-\frac{f(u)}{u^{p-1}}  \right)\big(k^p\, \varphi^p-u^p\big)\,dx\\
&\le \big(\lambda_{1, p}-\mu_\infty\big) \int_{\Omega}\big(k^p\, \varphi^p-u^p\big)\,dx
\end{split}
\]
giving \eqref{p1} by the positivity of the integral.  By the previous chain of inequalities,  $\mu_\infty=\lambda_{1,p}$ implies that $f(u)=\lambda_{1,p} u^p$ on $\Omega$, contradicting \eqref{pa}, so that the inequality in \eqref{pl} is strict.
Consider now 
\[
\mu_0:=\lim_{t\to0^+}\frac{f(t)}{t^{p-1}}.
\] 
By the properties of the first eigenvalue and the monotonicity of $t\mapsto f(t)/t^{p-1}$, there holds
\[
\lambda_{1,p}\int_\Omega u^p\,dx \leq\int_\Omega\vert\nabla u\vert^p\,dx\leq\int_\Omega f(u)\, u\,dx =\int_\Omega\frac{f(u)}{u^{p-1}}\, u^p\,dx\leq\mu_0\int_\Omega u^p\,dx.
\]
The latter chain of inequalities ensures that $\lambda_{1, p}\le \mu_0$ and that equality holds if and only if $f(u)=\mu_0\, u^p$, again contradicting \eqref{pa} and proving  that $\lambda_{1, p}<\mu_0$.

To prove the assertion on $J$, observe that the first inequality in \eqref{basicf} ensures by standard methods that $J$ is coercive on $W^{1,p}_0(\Omega)$ and possesses a minimum $\bar u\in W^{1,p}_0(\Omega)\cap C^{1,\alpha}(\overline{\Omega})$. Being $F$ odd, it holds $J(|v|)\le J(v)$ for any $v\in W^{1,p}_0(\Omega)$, so that we can assume that $\bar u\ge 0$.  Moreover, again by standard methods, the second inequality in \eqref{basicf} implies
\[
\inf_{W^{1,p}_0(\Omega)} J<0.
\] 
We conclude that $\bar u$ is nontrivial, and therefore strictly positive by the strong minimum principle. 

It remains to show that $u=\bar u$. Applying again Lemma  \ref{lem:monotonicity} we have
\[
0\leq \int_{\Omega}\left(\frac{f(u)}{u^{p-1}}- \frac{f(\overline{u})}{\overline{u}^{p-1}} \right)(u^p-\overline{u}^p)\, dx.
\]
By the monotonicity assumption, the two factors in the integrand have opposite sign, so that we infer
\begin{equation}
\label{peq}
\frac{f(u)}{u^{p-1}}=\frac{f(\bar u)}{\bar u^{p-1}}\qquad \text{in $\Omega$}
\end{equation}

 Next recall that the pointwise Picone inequality \cite{brasco-franzina}
\begin{equation}
\label{picone}
|\nabla v|^{p-2}\nabla v\cdot \nabla \frac{w^p}{v^{p-1}}\le |\nabla w|^p, \qquad v, w>0
\end{equation}
becomes an equality in a connected set if and only if $v=k\, w$. Applying the latter for $v=\bar u$ and $w=u$ and integrating, we get
\[
\int_\Omega \frac{f(\bar u)}{\bar u^{p-1}}\, u^p\, dx=\int_\Omega|\nabla \bar u|^{p-2}\nabla \bar u\cdot\nabla \frac{u^p}{\bar u^{p-1}}\le \int_\Omega |\nabla u|^p\le \int_\Omega f(u)\, u\, dx=\int_\Omega \frac{f(u)}{u^{p-1}}\, u^{p}\, dx.
\]
Since $u$ is positive, \eqref{peq}  implies that equality is attained in \eqref{picone}   for $ v=\bar u$ and $w=u$ everywhere in $\Omega$. Hence
\be\label{eq:prop}
u=k\, \bar u \,,\qquad\mbox{for some $k>0$.}
\ee
Let $g(t):=f(t)/t^{p-1}$ and assume that $k>1$ in \eqref{eq:prop}. By continuity, for any $n\in \N$ there exists $x_n\in \Omega$ such that 
\[
u(x_{n})=\| u \|_\infty/k^n
\]
and thus \eqref{peq} and \eqref{eq:prop} give
\[
g(\|u\|_\infty)=g(u(x_0))=g(\bar u(x_0))=g(u(x_0)/k)=g(u(x_1))=\dots=g(\|u\|_\infty/k^n)
\]
for any $n\ge 1$. Therefore $g$, being non-increasing, is constant on $(0, \|u\|_\infty)$, contradicting \eqref{pa}. Similarly, if $0<k<1$, we infer that $g$ is constant on 
\[
(0, \|\bar u\|_\infty)=(0, \|u\|_\infty/k)\supseteq (0, \|u\|_\infty), 
\] 
again contrary to \eqref{pa}.
Therefore $k=1$ and $u=\bar u$, as claimed.
 \end{proof}

\section{Proof of the main result}\label{sec4}
Our proof consists in a two-step approximation. Following \cite{Saka}, we first show that it suffices to deal with strictly convex domains. Then, under such assumption, we introduce a regularized problem for which we can prove the claim of Theorem \ref{Mth}. The solutions of the regularized problem converge uniformly to that of the original problem, thus proving Theorem \ref{Mth}. 

In light of the result of \cite{Saka}, we will suppose henceforth that $u$ is not a $\lambda_{1,p}$-eigenfunction, so that inequalities \eqref{basicf} are in place, ensuring that the functional $J$ in \eqref{defJ}
has a unique positive minimiser on $W^{1,p}_0(\Omega)$ by Proposition \ref{BO} and Lemma \ref{lemma21}. Notice that this holds for any smooth $\Omega$, justifying some of the following arguments.
\medskip

\subsection{Restricting to strictly convex domains}
We start by showing that there is no loss of generality considering strictly convex domains. To this aim, consider a sequence of strongly convex smooth domains $(\Omega_k)_{k\in\mathbb{N}}$ such that 
\[
\overline{\Omega}_k\subseteq\Omega_{k+1}\qquad\forall k\in\mathbb{N}\,,\qquad \bigcup_{k\in\mathbb{N}}\Omega_k=\Omega\,. 
\]
Let $u_k\in W^{1,p}_0(\Omega_k)$ be the unique solution to \eqref{eq:p-laplace} in $\Omega_k$, which can be extended to $\Omega$ setting $u_k=0$ on $\Omega\setminus\Omega_k$. 
Observe that $u_k$ is the unique positive minimiser of 
\[
J_k(z)=\int_{\Omega_k}\frac{1}{p}\, \vert\nabla z\vert^p-F(z)\,dx\,,
\]
and by minimality 
\[
J(u_{k+1})=J_{k+1}(u_{k+1})\leq J_{k+1}(u_k)=J(u_k)\,,
\]
and then the coercivity of $J$ ensures that $(u_k)_{k\in\mathbb{N}}$ is bounded in $W^{1,p}_0(\Omega)$. Thus, up to subsequences, we can assume that 
\[
  \begin{split}
  & u_n\rightharpoonup \tilde{u} \quad\mbox{weakly in $W^{1,p}_0(\Omega)$} \\
 & u_n\to \tilde{u}\quad \mbox{strongly in $L^s(\Omega)$, $\forall s\in[1,p]$} \\
 & u_n\to \tilde{u}\quad \mbox{a.e.}
  \end{split}
\]
  for some $\tilde{u}\in W^{1,p}_0(\Omega)$.
The continuity of the Nemitski operator associated to $F$  gives
\[
  \lim_{n\to\infty}\int_\Omega F(u_n)\,dx =\int_\Omega F(\tilde{u})\,dx\,.
 \]
  
  By semicontinuity of $v\mapsto \| \nabla v\|_p^p$ and by the minimising properties of $u_k$, we thus conclude that $\tilde{u}$ minimises $J$ and therefore, again by Proposition \ref{BO},  $\tilde{u}=u$. Then, assume that we proved the claim of Theorem \ref{Mth}   for $u_k$, for all $k\in\mathbb{N}$, that is, $\varphi(u_k)$ is concave on $\Omega_k$. The convexity functions $c_k$ of $\varphi(u_k)$ are therefore non-positive on $\Omega_k\times \Omega_k\times[0, 1]$ and converge pointwise a.\,e.\, on $\Omega\times\Omega\times[0, 1]$ to the convexity function $c$ of $\varphi(u)$. By the continuity of $u$, we thus infer that $c\le 0$ on $\Omega_{k_0}\times\Omega_{k_0}\times[0, 1]$ for any $k_0\in \N$, thus $\varphi(u)$ is concave on $\Omega$. Therefore we will suppose in the following that $\Omega$ is strictly convex.

\subsection{A regularized problem}

Let us now turn to the approximation procedure. We introduce a regularized problem for which we can prove the claim of Theorem \ref{Mth}. In the following we can assume that $u$ is not a $\lambda_{1, p}$ eigenfunction, so that \eqref{basicf} holds true.

Given $\eps>0$, consider the variational problem \be\label{eq:eps-min}
\inf_{v\in W^{1,p}_0(\Omega)}I_\eps(v)=:\lambda_\eps\,.
\ee
where
\[
I_\eps(v):=\frac{1}{p}\int_\Omega(\eps\, (G(v)^2)^{1/p}+\vert \nabla v\vert^2)^{p/2}\,dx-\int_\Omega F(v)\,dx\,
\]
for some odd $G\in C^1(\R)$ that will later be chosen. In the following we will let $g:=G'$.
\begin{lemma}\label{existregular}
Suppose that $f\in C^0(\R)$ is even, obeys \eqref{basicf} and $t\mapsto f(t)/t^{p-1}$ is non-increasing on $\R_+$. Then, for any sufficiently small $\eps>0$, problem \eqref{eq:eps-min} admits a non-negative, nontrivial solution $u_\eps\in W^{1,p}_0(\Omega)$
such that, for a fixed $\alpha>0$ $u_\eps\to u$ in $C^{1, \alpha}(\Omega)$
as $\eps\to 0$, where $u$ is the unique positive solution to \eqref{eq:p-laplace}.
\end{lemma}
\begin{proof}
The equi-coercivity of $I_\eps$ follows from the one of $J$ and the obvious inequality $I_\eps\ge J$. Furthermore, if $u$ is as in the statement, we have by \eqref{basicf}
\[
J(u)=\inf_{W^{1,p}_0(\Omega)} J<0,
\]
so that by continuity $I_\eps(u)<0$ for sufficiently small $\eps>0$. Thus \eqref{eq:eps-min} has a nontrivial solution, which is non-negative by $I_\eps(|v|)\le I_\eps(v)$.
 In order to prove the convergence of $u_\eps$ to $u$, first observe that  by dominated convergence it is readily checked that for any $w\in W^{1,p}_0(\Omega)$.
\be
\label{pconv}
\lim_{\eps\to 0^+} I_\eps(w)=J(w).
\ee
Since 
\[
J(u)\le J(u_\eps)\le I_\eps(u_\eps)= \lambda_\eps\le I_\eps(0) =0
\]
it follows that $u_\eps$ is bounded in $W^{1,p}_0(\Omega)$, uniformly for $0<\eps<1$.  Hence there exists $v\in W^{1,p}_0(\Omega)$ such that, up to subsequences, $u_{\eps_k}\rightharpoonup v$ weakly in $W^{1,p}_0(\Omega)$, as $k\to\infty$. 

From the weak lower semicontinuity of $J$, the minimality of $u_{\eps_k}$ and \eqref{pconv}, we get
\[
J(v)\le \varliminf_k J(u_{\eps_k})\le \varliminf_k I_{\eps_k}(u_{\eps_k})\le \varliminf_k I_{\eps_k} (u)=J(u)
\]
so that $v$ is a minimizer for $J$ as well, forcing $u=v$ by the uniqueness proved in Proposition \ref{BO}.
The previous display forces $J(u_{\eps_k})\to J(u)$ but, as already noted,  
\[
\int_\Omega F(u_{\eps_k})\, dx\to \int_\Omega F(u)\, dx
\]
so that we infer that $\|\nabla u_{\eps_k}\|_p^p\to \|\nabla u\|_p^p$, giving the strong convergence $u_{\eps_k}\to u$ in $W^{1,p}_0(\Omega)$ by uniform convextiy. A standard sub-subsequence argument allow to conclude that $u_\eps\to u$ in $W^{1,p}_0(\Omega)$.  Arguing as in \cite[Proposition 6.2]{Saka}, we infer from \eqref{growthF} and the uniform bound on $\|\nabla u_\eps\|_{p}$ a uniform bound on $\|u_\eps\|_\infty$ for $\eps\in (0, 1)$. Standard regularity theory then ensures that $(u_\eps)$ is uniformly bounded in $C^{1, \beta}(\overline{\Omega})$, $\beta>0$ and thus converges to $u$ in $C^{1, \beta/2}(\overline{\Omega})$ by Ascoli-Arzel\'a theorem.
\end{proof}

We collect next some regularity properties of the minimisers $u_\eps$ just constructed.

\begin{lemma}\label{lemmareg}
Let $u_\eps$ be the solutions constructed in Lemma \ref{existregular} under its assumptions, and suppose that $G>0$ on $\R_+$. 
\begin{enumerate}
\item
Each $u_\eps$ satisfies in $\Omega$ the Euler-Lagrange equation
\be\label{eq:eps-eq}
-\Div\Big(\big(\eps\, G(u_\eps)^{\frac{2}{p}}+\vert\nabla u_\eps\vert^2\big)^{\frac{p-2}{2}}\nabla u_\eps\Big)= f(u_\eps)-\frac{\eps}{p}\, \big(\eps\, G(u_\eps)^{\frac{2}{p}}+\vert\nabla u_\eps\vert^2\big)^{\frac{p-2}{2}}G(u_\eps)^{\frac{2-p}{p}}g(u_\eps).
\ee
\item
For sufficiently small $\eps$, we have 
\[
u_\eps>0\  \text{in $\Omega$},\qquad  \frac{\partial u_\eps}{\partial n}>0 \ \text{on $\partial\Omega$}
\]
where $n$ is the interior normal to $\partial\Omega$.
\item
If $f$ and $g$ are $\alpha$-H\"older continuous, $u_\eps\in C^{2, \alpha}(\Omega)$ and there exists $\eta_0>0$ such that $u_\eps$ is  uniformly  bounded in $C^{2, \alpha}(S_{\eta_0})$, where   
\[
S_\eta:=\{x\in \Omega:\eta/2\le {\rm dist}(x, \partial\Omega)\le 2\, \eta\}
\]
\end{enumerate}
\end{lemma}

\begin{proof}
The first assertion is a simple calculation. Recall that $u$ itself satisfies 
\[
u>0 \ \text{in $\Omega$},\qquad  \frac{\partial u}{\partial n}>0 \ \text{on $\partial\Omega$}
\]
by the Hopf lemma, so that the second assertion follows from the convergence $u_\eps\to u$ in $C^{1,\alpha}(\overline{\Omega})$.
Regarding the last regularity property, it suffice to observe that the operator on the left hand side of \eqref{eq:eps-eq} is uniformly elliptic with H\"older continuous coefficients since, from the previous point, there exists $c_0>0$ such that either $|\nabla u_\eps|\ge c_0$ near $\partial\Omega$, or $F(u)\ge c_0$ away from $\partial\Omega$. The right hand side of \eqref{eq:eps-eq} is H\"older continuous away from $\partial\Omega$, so that the nonlinear  Schauder estimates give the claims in {\em (3)}.
\end{proof}

 Theorems  \ref{Mth} and \ref{Sth} will now follow from the next statement, through a simple approximation argument on the concavity function.
 \begin{proposition}\label{final}
 Under the assumptions of Theorem \ref{Mth}, let $\varphi$ be defined in \eqref{defvarphi}. For any sufficiently small $\delta>0$, there exist $\eps_\delta$ such that if $\eps<\eps_\delta$ then $v=\varphi(u_\eps)$ is concave on 
\[
\Omega_\delta=\{x\in \Omega:{\rm dist}(x, \partial\Omega)>\delta\}.
\]
The same statement holds for $v=\log u$.
\end{proposition}

 \subsection{Non-positivity at the boundary}
 
From the arguments given in the proof of Lemma \ref{lemmareg} we can fix $\delta_0$ so that the solution  $u$ of \eqref{eq:p-laplace} belongs to $C^2(\Omega\setminus\Omega_{\delta_0})$. Moreover, \eqref{assu} hold true for $u$ and, thanks to Lemma \ref{lemmavarphi} (and a direct computation in the case of $\log t$), Proposition \ref{propKorevaar} applies for both the transformations $\varphi$ defined in \eqref{defvarphi} an $\tilde{\varphi}(t)=\log t$. In the next argument, with $v$ we denote for simplicity   both $\varphi(u)$ and $\tilde{\varphi}(u)$, and $v_\eps$ will stand for both $\varphi(u_\eps)$ and $\tilde{\varphi}(u_\eps)$.  It follows that \eqref{k1} and \eqref{k2} both hold for $v$ and for any $\delta\le  \delta_1\le \delta_0$. 

If $\eta_0$ is given in {\em (3)} of Lemma \ref{lemmareg}, we next choose $\delta<\min\{\eta_0, \delta_1\}/2$ so that $v_\eps$ is uniformly bounded in $C^{2, \alpha}( S_\delta)$, and therefore converges to $u$ in $C^2( S_\delta)$. 
Since $u_\eps$ is uniformly bounded from below by a positive constant and both $\varphi$ and $\tilde{\varphi}$ are $C^2(\R_+)$, we infer that $v_\eps\to v$ in $C^2(S_\delta)$. Therefore \eqref{k1} continues to hold for any sufficiently small $\eps$, so that $v_\eps$ is locally strictly concave on $S_\delta$. 

We claim that \eqref{k2} holds as well for any (possibly smaller) $\eps$. Suppose not, so that there are
\[
\eps_n\to 0, \qquad  \bar x_n\in \partial\Omega_\delta,\qquad  x_n\in \Omega_\delta\setminus\{\bar x_n\}
\]
such that
\be
\label{condL}
L_{v_{\eps_n}, \bar x_n}(x_n)\ge v_{\eps_n}(x_n).
\ee
By compactness we can assume that $\bar x_n\to \bar x$, $x_n\to x\in\overline\Omega_\delta$. By eventually lowering $\delta_0$, we see by the smoothness and  strong convexity of $\Omega$ that there is a constant $c=c(\delta, \Omega)>0$ such that
\[
 B_{c\, \delta}(\bar x)\subseteq S_\delta, \qquad \forall \bar x\in \partial\Omega_\delta.
 \]
Since $v_{\eps_n}$ is strictly concave on $B_{c\, \delta}(\bar x_n)$, the points $x_n$ fulfilling \eqref{condL} must be at least $c\, \delta$ away from $\bar x_n$. Therefore the limit point $x$ must be at least $c\, \delta$ away from $\bar x$. By taking the limit in \eqref{condL} we find that
\[
L_{v, \bar x}(x)\ge v(x), \qquad x\ne \bar x
\]
contradicting the established validity of \eqref{k1} for $v$.
Therefore both \eqref{k1} and \eqref{k2} hold for $v_\eps$, if $\eps$ is sufficiently small, and the last statement of Proposition \ref{propKennington} ensures that the convexity function of $v_\eps$ cannot assume positive values on $\partial(\Omega_\delta\times\Omega_\delta)$.

We next show that, for any $\delta$ found in the previous point and any correspondingly small $\eps$, the convexity function of $\varphi(u_\eps)$ (or $\log u_\eps$, in the proof of Theorem \ref{Sth}) cannot have a positive interior maximum on $\Omega\times\Omega$. In order to do so, we will choose accordingly the function $G$ in the approximation procedure given in \eqref{eq:eps-min} (which, so far, played no r\^ole).
 
  \subsection{The transformed equation}
 We first consider the the number 
 \[
 M=\inf\big\{t>0:f(t)=0\big\}.
 \]
 In both   Theorems \ref{Sth} and \ref{Mth}, the function $t\mapsto f(t)/t^{p-1}$ is non-increasing (by Lemma \ref{lemma21} in the second case), hence
  $f$ vanishes identically on $[M, +\infty)$. Clearly we may assume that  $M>0$, for otherwise problem \eqref{eq:p-laplace} have no solutions at all. We thus assume that
  \be
  \label{assG}
  G(t)=G(M)\qquad \forall t\ge M
  \ee
 and claim that, under this assumption,
 \be
 \label{cl}
 \sup_\Omega u_\eps\le M.
 \ee
 Indeed, if this is not the case, the function $w_\eps=u_\eps-M$ is non-negative on the open set $\{u_\eps>M\}$, vanishes on its boundary, and attains a positive maximum inside. By \eqref{assG} and $f(t)=0$ for all $t\ge M$, we infer from \eqref{eq:eps-eq} that $w_\eps$ solves
 \[
 -\Div\Big(\big(\eps\, G(M)^{\frac{2}{p}}+\vert\nabla w_\eps\vert^2\big)^{\frac{p-2}{2}}\nabla w_\eps\Big)=0.
 \]
The operator on the left hand side  is monotone, therefore by comparison we obtain $w_\eps\equiv 0$ and thus a contradiction, proving \eqref{cl}. Notice that the truncation prescribed in \eqref{assG}, despite lowering the regularity of $G$, doesn't affect the validity of {\em (3)} of Lemma \ref{lemmareg}, due to \eqref{cl}.

For a, still to be determined, increasing diffeomorphism $\varphi\in C^1(\R_+,\R)$  with inverse $\psi$, set
\be
\label{condG}
G(t)=\big(\psi'(\varphi(t))\big)^p, \qquad t\in (0, M)
\ee
so that the oddness condition together with \eqref{assG} defines $G$ on the whole $\R$.
For the corresponding solutions $u_\eps$ of \eqref{eq:eps-eq} we set
  \[
  v_\eps:=\varphi(u_\eps)\,,
  \]
  and compute the equation solved by $v_\eps$. By construction, $\psi$ satisfies
  \be
  \label{f1}
  G(\psi(s))=\psi'(s)^p, \qquad g(\psi(s))=p\, \psi'(s)^{p-2}\, \psi''(s)
  \ee
 for $s\in v_\eps(\Omega)$ and moreover $\nabla u_\eps=\psi'(v_\eps)\nabla v_\eps$.
It follows that
  \[
  \eps\,  G(u_\eps)^{2/p}+\vert\nabla u_\eps\vert^2=\psi'(v_\eps)^2(\eps+|\nabla v_\eps|^2)
  \]
so that
  \be
  \label{eq1}
 \frac{\eps}{p}\, (\eps\, G(u_\eps)^{\frac{2}{p}}+\vert\nabla u_\eps\vert^2)^{\frac{p-2}{2}}G(u_\eps)^{\frac{2-p}{p}}g(u_\eps)=  \eps\, (\eps+|\nabla v_\eps|^2)^{\frac{p-2}{2}}\psi'(v(\eps)^{p-2} \psi''(v_\eps).
  \ee
Regarding the left hand side of \eqref{eq:eps-eq}, from
\[
D^2u_\eps=\psi''(v_\eps)\, \nabla v_\eps\otimes\nabla v_\eps+\psi'(v_\eps)\, D^2v_\eps
\]
we compute
  \be
  \label{eq2}
  \begin{split}
  &\Div \Big(\big(\eps\, G(u_\eps)^{2/p}+\vert\nabla u_\eps\vert^2\big)^{\frac{p-2}{2}}\nabla u_\eps\Big) =\\
  & \psi'(v_\eps)^{p-1}\Div\Big( \big(\eps+\vert \nabla v_\eps\vert^2 \big)^{\frac{p-2}{2}}\nabla v_\eps\Big) +(p-1)\, \psi'(v_\eps)^{p-2}\psi''(v_\eps)\left( \eps+\vert\nabla v_\eps\vert^2\right)^{\frac{p-2}{2}}\vert\nabla v_\eps\vert^2.
  \end{split}
  \ee
 Putting \eqref{eq1} and \eqref{eq2} into \eqref{eq:eps-eq}, we obtain
  \be
  \label{eg}
 - \Div\Big(\big(\eps+\vert \nabla v_\eps\vert^2 \big)^{\frac{p-2}{2}}\nabla v_\eps\Big) = \frac{f(\psi(v_\eps))}{(\psi'(v_\eps))^{p-1}}+  \frac{\psi''(v_\eps)}{\psi'(v_\eps)} \left( \eps+\vert\nabla v_\eps\vert^2\right)^{\frac{p-2}{2}}\big((p-1)\, \vert\nabla v_\eps\vert^2 -\eps\big).
 \ee
 Now we split the proof in two cases, choosing $\psi$ (or, equivalently, $G$) accordingly.
 
 \medskip
 \noindent
 \subsection{Non positivity in the interior}
 We finally choose the transformation. For the first statement of Proposition \ref{final}, let  $\varphi$  defined by \eqref{defvarphi}, which then fulfils \eqref{condG}  on $[0, M]$ for $G=F$, i.\,e.
   \be
  \label{f1}
  \psi'(s)=F^{1/p}(\psi(s))
  \ee
 Differentiating this relation we obtain 
\[
 \frac{f(\psi(s))}{(\psi'(s))^{p-1}}=p\, \frac{\psi''(s)}{\psi'(s)},
 \]
 which, inserted into \eqref{eg} gives 
\[
  \begin{split}
  -\Div\Big(\big(\eps+\vert\nabla v_\eps\vert^2 \big)^{\frac{p-2}{2}}\nabla v_\eps  \Big)&= \frac{\psi''(v_\eps)}{\psi'(v_\eps)}\, b(\nabla v_\eps)
  \end{split}
 \]
  where
  \be\label{eq:b}
  b(\nabla v_\eps):= p+\left((p-1)\, |\nabla v_\eps|^2-\eps\right)\left(\eps+\vert\nabla v_\eps\vert^2 \right)^{\frac{p-2}{2}}. 
  \ee
 Consider the function
 \[
 h(t)=p+\left((p-1)\, t-\eps\right)\left(\eps+t\right)^{\frac{p-2}{2}}, \qquad t\ge 0
 \]
 An elementary computation shows that $h$ is increasing, so that its minimum is $h(0)=p-\eps^{p/2}$. It follows that the function $b$ defined in \eqref{eq:b} is positive whenever $\eps<p^{2/p}$, which we will assume henceforth.
 
In order to apply Proposition \ref{propKennington}, we have to check that 
\be
\label{clf}
\frac{\psi''}{\psi'}\quad \text{is  non-increasing and harmonic concave on $v_\eps(\Omega)$ }.
\ee
From \eqref{cl} we infer $\sup_\Omega v_\eps\le \varphi(M)$
and clearly $\psi(s)>0$ for all $s\in v_\eps(\Omega)$, since $0<u_\eps=\psi(v_\eps)$. Therefore 
\be
\label{f5}
\psi(v_\eps(\Omega))\subseteq [0, M].
\ee

For the first assertion in \eqref{clf}, recall from \eqref{f1} that
\[
 \frac{\psi''}{\psi'}=\frac{(F^{1/p}\circ \psi)'}{\psi'}=(F^{1/p})'\circ \psi.
 \]
Since $F^{1/p}$ is assumed to be  concave on $(0, M)$ and \eqref{f5} holds true,   this expression is   non-increasing being the composition of a  non-increasing function with an increasing one.
  
Finally, we claim that $\psi'/\psi''$ is convex.  We first compute
  \[
  \frac{\psi'(s)}{\psi''(s)}= p\, \frac{F(\psi(s))}{f(\psi(s))} \, F(\psi(s))^{-1/p}.
  \]
The function $F/f$ is assumed to be convex on $(0, M)$ and then it is differentiable, except at most at countable set $A\subseteq [0, M]$. If $B=\varphi(A)$ and $s\in v_\eps(\Omega)\setminus B$, then $\psi(s)\in [0, M]\setminus A$ due to \eqref{f5}. Thus  $F/f$ is differentiable at $\psi(s)$ and it holds
  \[
  \begin{split}
  \left(  \frac{\psi'(s)}{\psi''(s)} \right)'&=p\, \left(\frac{F(\psi(s))}{f(\psi(s))} \right)' \, F^{-1/p}(\psi(s))-\frac{F(\psi(s))}{f(\psi(s))} \, F^{-1-1/p}(\psi(s))f(\psi(s))\psi'(s)\\
  &=p\, \left(\frac{F(\psi(s))}{f(\psi(s))} \right)' \frac{1}{\psi'(s)} -1
  \end{split}
  \]
where we used \eqref{f1} in the last step. Therefore
\[
 \left(  \frac{\psi'}{\psi''} \right)'= p\, \left(\frac{F}{f} \right)' \circ \psi -1\qquad \text{on  $v_\eps(\Omega)\setminus B$}
 \]
which is non-decreasing as composition of non-decreasing functions. Since $B$ is at most countable, the convexity of $\psi'/\psi''$ follows by well-known characterisations of convex function of the real line.

Hence Proposition \ref{propKennington} applies, and the convexity function of $\varphi(u_\eps)$ cannot attain a positive maximum on $\Omega\times\Omega$ and {\em a fortiori} on $\Omega_\delta\times\Omega_\delta$.
 
 \medskip
 
Regarding the second statement of Proposition \ref{final}, we choose in \eqref{eg} the function $G(t)=t^p$,
 which corresponds to $\psi(s)=e^s$, $\varphi(t)=\log t$. In this case \eqref{eg} reads
\[
 \Div\Big(\big(\eps+\vert \nabla v_\eps\vert^2 \big)^{\frac{p-2}{2}}\nabla v_\eps\Big) = \frac{f\big(e^{v_\eps}\big)}{e^{(p-1)v_\eps}}+   \left( \eps+\vert\nabla v_\eps\vert^2\right)^{\frac{p-2}{2}}\big((p-1)\, \vert\nabla v_\eps\vert^2 -\eps\big)
 \]
and the condition required in Theorem \ref{Mth} allow to apply Proposition \ref{propKennington} to exclude the positivity of the convexity function of $v_\eps$ in $\Omega\times\Omega$.


\end{document}